\documentclass{ifacconf}

\usepackage{graphicx}      % include this line if your document contains figures
\graphicspath{{figures/}} % path for figures etc
\usepackage{natbib}        % required for bibliography

\usepackage{arydshln}
\usepackage{amsmath}
\usepackage{mathdots}
\usepackage{mathrsfs}  % for \mathscr
\usepackage{amssymb} % symbols as \prec \succ etc
\usepackage{arydshln}% dashed horizontal and vertical lines
\usepackage{graphicx} % graphic support
\usepackage{xcolor}
\usepackage{tikz}
\usepackage[absolute]{textpos} % For "Watermark" on first page
\usepackage{mathtools}

\newcommand{\R}{\mathbb{R}}         % real numbers
\newcommand{\C}{\mathbb{C}}         % complex numbers
\renewcommand{\S}{\mathbb{S}}          % symmetric matrices
\newcommand{\Ls}{\mathscr{L}}

\newcommand{\Vb}{\mathbf{V}}
\newcommand{\vb}{\mathbf{v}}
\newcommand{\Pb}{\mathbf{P}}

\newcommand{\cge}{\succcurlyeq}
\newcommand{\cle}{\preccurlyeq}
\newcommand{\cl}{\prec}
\newcommand{\cg}{\succ}
\newcommand{\kron}{\otimes}         % kronecker product
\newcommand{\Delf}{\mathbf{\Delta}}
\newcommand{\rhi}{\mathrm{RH}_\infty}
\newcommand{\rli}{\mathrm{RL}_\infty}

\newcommand{\diag}{\mathrm{diag}}

\newcommand{\eig}{\mathrm{eig}}
\newcommand{\Del}{\Delta}
\newcommand{\del}{\delta}
\newcommand{\ga}{\gamma}

\newcommand{\om}{\omega}
\newcommand{\io}{i\omega}

\newcommand{\vect}{\mathrm{vec}}
\newcommand{\mat}[2]{\left(\begin{array}{@{}#1@{}}#2\end{array}\right)} %ersetze #1 durch @{}#1@{} fuer weniger platz rechts und links
\newcommand{\smat}[1]{\left(\begin{smallmatrix}#1\end{smallmatrix}\right)}

\newcommand{\teq}[1]{\quad\text{#1}\quad} % for text in math enironments
\newenvironment{red_test}{\color{red}}{} % mark red

\renewcommand{\t}{\tilde}

\newcommand{\mstrut}[1]{\rule{0pt}{#1}}
\let\oldhline=\hline % remember original command
\renewcommand{\hline}{\oldhline\mstrut{2.5ex}}
\let\oldhdashline=\hdashline % remember original command
\renewcommand{\hdashline}{\oldhdashline\mstrut{2.5ex}}

\newtheorem{theo}{Theorem}%[section]
\newtheorem{lemm}[theo]{Lemma}
\newtheorem{coro}[theo]{Corollary}
\newtheorem{defi}[theo]{Definition}

\newenvironment{proof}[1][Proof]{%\vspace{1.5ex}
	\bf #1. \rm}%{}%
{\hfill \footnotesize{$\blacksquare$}\vspace{2ex}}

\begin{document}
	\begin{frontmatter}
		
		\title{A Dynamic S-Procedure for Dynamic Uncertainties\thanksref{footnoteinfo}}
		% Title, preferably not more than 10 words.
		
		\thanks[footnoteinfo]{Funded by Deutsche Forschungsgemeinschaft (DFG, German Research Foundation) under Germany’s Excellence Strategy - EXC 2075 - 390740016. We acknowledge the support by the Stuttgart Center for Simulation Science (SimTech).}
		
		\author[First]{Tobias Holicki and Carsten W. Scherer}
		
		\address[First]{Department of Mathematics, University of Stuttgart, Germany,
			e-mail: \{tobias.holicki, carsten.scherer\}@imng.uni-stuttgart.de
		}

		\begin{abstract}                % Abstract of not more than 250 words.
			We show how to compose robust stability tests for uncertain systems modeled as linear fractional representations and affected by various types of dynamic uncertainties. Our results 
			are formulated in terms of linear matrix inequalities and
			rest on the recently  established notion of finite-horizon integral quadratic constraints with a terminal cost. 
			The construction of such constraints is motivated by an unconventional application of the S-procedure in the frequency domain with dynamic Lagrange multipliers. Our technical contribution reveals how this construction can be performed by dissipativity arguments in the time-domain and in a lossless fashion.
			This opens the way for generalizations to time-varying or hybrid systems.								
			
		\end{abstract}

		
	\end{frontmatter}
	%===============================================================================
	
	%\begin{textblock}{9}(1.15, 16)
	%	This work has been submitted to IFAC for possible publication.
	%\end{textblock}
	
%	\begin{textblock}{13}(1.55, 15.75)
%		\fbox{
%		\small \textcopyright~ 2022 the authors. This work has been accepted to IFAC for publication under a Creative Commons Licence CC-BY-NC-ND}
%	\end{textblock}
	
	\begin{textblock}{12}(1.75, 15.75)
		\fbox{
			\begin{minipage}{\textwidth}
				\small \textcopyright~ 2022 IFAC. This work has been published in IFAC-PapersOnline under a Creative Commons Licence CC-BY-NC-ND. \\ DOI: 10.1016/j.ifacol.2022.09.331 \hfill
				DOI for corresponding Code: 10.24433/CO.3032988.v1 
			\end{minipage}
		}
	\end{textblock}

	%-----------------------------------------------------------------------------------------
	\section{Introduction}
	%-----------------------------------------------------------------------------------------

	The S-procedure (also referred to as S-lemma or S-method)
	was developed more than 90 years ago and has applications in optimization, control and many other areas in mathematics.
	Naturally, by now there exists a multitude of variations of the S-procedure, some of which can be found, e.g., in the books by \cite{BoyGha94}, \cite{BenNem01} or by \cite{BoyVan04}.
	Detailed surveys on results related to the S-procedure can be found in \citep{Uhl79}, \citep{PolTer07} and in \citep{GusLik06}.

	Traditionally, the S-procedure guarantees the positivity of a quadratic function on a set defined by quadratic inequalities through the positivity of a related quadratic function. The latter function is constructed by Lagrange relaxation and, hence, involves sign-constrained Lagrange multipliers that are also called scalings in the context of control.
	In the book by \cite{BoyGha94}, for example, it is shown how this standard variant can be used for the robustness analysis of uncertain dynamical systems. However, this approach is limited since it only permits the construction of robustness tests with scalings that are real-valued and structured. In contrast, e.g., $\mu$-theory \citep{PacDoy93} offers the use of frequency-dependent scalings for dynamic and parametric uncertainties in order to substantially reduce the conservatism of such tests.
	Similarly, the so-called full-block variant of the S-procedure in \citep{Sch97, Sch01} permits the reduction of conservatism by relying on unstructured scalings, even if the uncertainty itself is structured.
	
	Robust stability and performance tests based on the full-block S-procedure and most variants thereof can be viewed as an application of integral quadratic constraint (IQC) analysis with static (frequency-independent) multipliers \citep{MegRan97}. Also in this context, it is well-established that dynamic (frequency-dependent) multipliers often lead to less conservative results. Recently, \cite{Sch21} has hinted at how to construct finite time-horizon IQCs with dynamic multipliers for one dynamic uncertainty and with frequency domain arguments.
	
	Based on these ideas, the purpose of this paper is to propose a dynamic version of the S-procedure which has not found much attention in the literature. In particular, we provide a new alternative and simplified proof of the aforementioned result of \cite{Sch21} without relying on frequency domain arguments. This is beneficial, e.g., for its generalization to hybrid systems, where no suitable notion of a frequency domain exists.
	Moreover, we demonstrate that our conceptual proof allows for deriving several novel variations and extensions that permit us, e.g., to analyze the robustness of a system against dynamic uncertainties with a Nyquist plot known to be located in a given LMI region.
	This substantially extends, for example, the inspiring results by \cite{Bal02}.

	\vspace{1ex}
	%-----------------------------------------------------------------------------------------
	\noindent\textit{Outline.} %
	%-----------------------------------------------------------------------------------------
	%
	The remainder of the paper is organized as follows. After a short paragraph on notation, we specify the considered robust analysis problem and recall some preliminary results in Section \ref{DSP::sec::preliminaries}. In Section \ref{DSP::sec::main}, we provide and elaborate on our main analysis result for systems affected by a class of dynamic uncertainties.
	In Section~\ref{DSP::sec::variations}, we demonstrate how to modify our approach in order to handle dynamic uncertainties whose frequency response is located in regions of the complex plane beyond disks and half-spaces.
	We conclude with a numerical example and some further remarks in Sections \ref{DSP::sec::example} and \ref{DSP::sec::conclusions}, respectively. Finally, most proofs are moved to the appendix.

	\vspace{1ex}
	%-----------------------------------------------------------------------------------------
	\noindent\textit{Notation.} %
	%-----------------------------------------------------------------------------------------
	%
	Let
	$L_2^n\! :=\! \{x \!\in\! L_{2e}^n:  \int_0^\infty \!x(t)^T\! x(t) \,dt\!<\! \infty \}$ where $L_{2e}^n$ is the space of locally square integrable functions $x:[0, \infty) \to \R^n$.
	Moreover, $\S^n$ denotes the set of symmetric $n\times n$ matrices, and $\rhi^{m \times n}$ ($\rli^{m\times n}$) is the space of real rational $m \times n$ matrices without poles in the extended closed right half-plane (imaginary axis) and is equipped with the maximum norm $\|\cdot\|_\infty$.
	%\red{Elements in $\rhi^{m \times n}$ are said to be stable.}
	For a transfer matrix $G(s) = D + C(sI - A)^{-1}B$, we take $(-A^T,  C^T,  -B^T,  D^T)$ as a realization of $G^\ast(s)=G(-s)^T$.
	Finally, we use the abbreviation
	\begin{equation*}
		\arraycolsep=1pt
		\diag(X_1, \dots, X_N) := \smat{X_1 & & 0 \\[-1ex] & \ddots & \\ 0 & & X_N}
	\end{equation*}
	for matrices $X_1, \dots, X_N$,
	utilize the Kronecker product $\kron$ as defined in \citep{HorJoh91}
	and indicate objects that can be inferred by symmetry or are not relevant with the symbol ``$\bullet$''.
	
	%-----------------------------------------------------------------------------------------
	\section{Preliminaries}\label{DSP::sec::preliminaries}
	%-----------------------------------------------------------------------------------------
	
	For real matrices $A, B, C, D$ and some initial condition $x(0) \in \R^n$, let us consider the uncertain feedback interconnection
	\begin{subequations}
		\label{DSP::eq::interconn}
		\begin{equation}
			\mat{c}{\dot x(t) \\ z(t)} = \mat{cc}{A & B \\ C& D} \mat{c}{x(t) \\ w(t)},\quad
			w(t) = \Del(z)(t)
			\label{DSP::eq::sys}
		\end{equation}
		for $t \geq 0$. The uncertainty $\Del$ is assumed to be a stable LTI system described by its  input-output map $\Del: L_{2e}^k \to L_{2e}^l$,
		\begin{equation}
			\Del(z)(t)
			=\int_{0}^t C_\Del e^{A_\Del(t - s)}B_\Del z(s)~ds + D_\Del z(t)
			\label{DSP::eq::uncertainty}
		\end{equation}
		for some real matrices $A_\Del, B_\Del, C_\Del, D_\Del$.
		We slightly abuse the notation by using the same symbol for the map \eqref{DSP::eq::uncertainty} and the transfer matrix ${\Del(s)= C_\Del (s I - A_\Del)^{-1}B_\Del + D_\Del}$.
		We suppose that $\Delta$ satisfies some structural constraints, which are imposed on its frequency response as
		\begin{equation}
			\Del(\io) \in \Vb \teq{ for all } \om \in \R \cup \{\infty \}
			\label{DSP::eq::uncertainty_set}
		\end{equation}
	\end{subequations}
	for some given value set $\Vb \subset \C^{l \times k}$. We denote the set of all such dynamic uncertainties by $\Delf(\Vb)$.
	
	\begin{defi}%-----------------------------------------------------------------------------
		The interconnection \eqref{DSP::eq::interconn} is said to be robustly stable if it is well-posed, i.e., if $\det(I - DV) \neq 0$ holds for all $V \in \Vb$, and if there exists a constant $M > 0$ such that the state-trajectories $x(\cdot)$ satisfy
		\begin{equation*}
			\lim_{t \to \infty} x(t) = 0
			\teq{ and }
			\|x(t)\| \leq M\|x(0)\|
			\text{ for all }t \geq 0
			%\|x(t)\| \leq Me^{-\ga t}\|x(0)\| \teq{ for all }t\geq 0
			%x \in L_2^n \teq{ and } \lim_{t \to \infty} x(t) = 0
		\end{equation*}
		for all initial conditions $x(0) \in \R^n$ and any $\Del \in \Delf(\Vb)$.
	\end{defi}%-------------------------------------------------------------------------------

	The following robust analysis result is a consequence of the (full-block) S-procedure as proposed by \cite{Sch97}.

	\begin{lemm}%-----------------------------------------------------------------------------
		\label{DSP::lem::fb_s_proc}
		The interconnection \eqref{DSP::eq::interconn} is robustly stable if there exist some (Lyapunov) certificate $X \in \S^n$ and a multiplier $P\in \S^{k+l}$ satisfying the LMIs
		\begin{subequations}
			\label{DSP::lem::eq::fb_s_proc_lmi}
			\begin{equation}
				X \cg 0,
				\label{DSP::lem::eq::fb_s_proc_lmia}
			\end{equation}
			\begin{equation}
				(\bullet)^T \mat{cc}{0 & X \\ X & 0} \mat{cc}{I & 0 \\ A & B} + (\bullet)^T P \mat{cc}{C & D \\ 0 & I} \cl 0,
				\label{DSP::lem::eq::fb_s_proc_lmib}
			\end{equation}
			\begin{equation}
				\mat{c}{0 \\ I}^T P \mat{c}{0 \\ I} \cle 0
				\label{DSP::lem::eq::fb_s_proc_extra_dynamic}
			\end{equation}
			and
			\begin{equation}
				\label{DSP::lem::eq::fb_s_proc_lmic}
				\mat{c}{I \\ V}^\ast P \mat{c}{I \\ V} \cge 0
				\teq{ for all }V \in \Vb.
			\end{equation}
		\end{subequations}
	\end{lemm}%-------------------------------------------------------------------------------

	The inequalities \eqref{DSP::lem::eq::fb_s_proc_lmi} are not directly numerically tractable since \eqref{DSP::lem::eq::fb_s_proc_lmic} usually involves an infinite number of LMIs.
	As a remedy one typically does not search for a suitable multiplier $P$ with
	\eqref{DSP::lem::eq::fb_s_proc_lmic} in the full set $\S^{k+l}$,
	but one confines the search to a subset $\Pb(\Vb)\subset \S^{k+l}$ which admits an
	LMI representation. This means that there exist affine symmetric-valued maps $F$ and $G$ on $\R^\bullet$ with $\Pb(\Vb)=\left\{F(\nu)\,\middle|\ \nu \in \R^{\bullet} \text{ and }G(\nu)\cg 0\right\}$.
	In order to counteract the conservatism introduced by this restriction while maintaining computational efficiency, one should aim for subsets that capture the properties of $\Vb$ as well and as simple as possible.
	A detailed discussion and several choices for concrete instances of $\Vb$ are found, e.g., in \citep{SchWei00}.

	First, let us assume that $\Delta$ is a single repeated dynamic uncertainty, i.e., $\Del = \del I_k$ for some $\del \in \rhi^{1 \times 1}$. Moreover, we suppose that its values on the imaginary axis are confined to the set
	\begin{equation}
		\label{DSP::eq::concrete_value_set}
		\Vb  := \left\{v I_k ~\middle|~ v \in \C \text{ and }\mat{c}{1 \\ v}^\ast P_0 \mat{c}{1 \\ v} \geq 0 \right\}
	\end{equation}
	for some given matrix $P_0 = \smat{q & s \\s & r}\in \S^2$ with $r \leq 0$.
	This means that the Nyquist plot of $\del$ is contained in the disk or half-plane corresponding to the matrix $P_0$.

	Note that $r\leq 0$ is a technical requirement in our proof, and it implies that the set $\Vb$ is convex; this restriction and the related inequality \eqref{DSP::lem::eq::fb_s_proc_extra_dynamic} in Lemma \ref{DSP::lem::fb_s_proc} can be dropped if the uncertainty $\Del$ affecting the system \eqref{DSP::eq::sys} is known to be parametric.
	In view of the KYP lemma \citep{Ran96}, the constraint \eqref{DSP::eq::uncertainty_set} with \eqref{DSP::eq::concrete_value_set} can equivalently be expressed in the time domain by means of a nonstrict dissipation inequality; in the sequel, we rather use the frequency domain formulation for brevity.

	The concrete description \eqref{DSP::eq::concrete_value_set} leads to the following tractable robust stability test.
	
	\begin{coro}%-----------------------------------------------------------------------------
		\label{DSP::coro::dynamic_rep_static}
		The interconnection \eqref{DSP::eq::interconn} with \eqref{DSP::eq::concrete_value_set} is robustly stable if there exist some certificate $X \in \S^n$ and a multiplier
		\begin{equation*}
			P \in \Pb(\Vb) := \left\{ P_0 \kron M~\middle|~M \in \S^{k},\ M \cg 0\right\}
		\end{equation*}
		satisfying \eqref{DSP::lem::eq::fb_s_proc_lmia} and \eqref{DSP::lem::eq::fb_s_proc_lmib}.
	\end{coro}%-----------------------------------------------------------------------------
	
	\begin{proof}
		We just have to observe that, by the rules of the Kronecker product,
		\begin{multline*}
			0 \cle \left(\mat{c}{1 \\ v}^{\!\ast}\! P_0\! \mat{c}{1 \\ v}\right) \kron M
			= (\bullet)^\ast (P_0 \kron M) \left( \mat{c}{1 \\v} \kron I_k\right)\\
			= (\bullet)^\ast (P_0 \kron M) \mat{c}{I \\vI_k}
			= (\bullet)^\ast (P_0 \kron M) \mat{c}{I \\ V}
		\end{multline*}
		%	\begin{equation*}
		%		(\bullet)^\ast P\! \mat{c}{I \\ V} = (\bullet)^\ast \!\mat{cc}{qM & sM \\ sM & rM} \!\mat{c}{I \\ v I}
		%%
		%		= M \cdot (\bullet)^\ast P_0 \mat{c}{1 \\ v} \cge 0
		%	\end{equation*}
		holds for any $V \!=\!vI_k \in \Vb$ and any $P\!=\!P_0 \!\kron\! M \in \Pb(\Vb)$.
	\end{proof}
	
	Note that neither Lemma \ref{DSP::lem::fb_s_proc} nor Corollary \ref{DSP::coro::dynamic_rep_static} takes into account that the uncertainty $\Del$ entering \eqref{DSP::eq::sys} is an LTI system \eqref{DSP::eq::uncertainty}. Consequently, both analysis criteria can be rather conservative, as emphasized for example by \cite{PooTik95}.
	This can be resolved, e.g., by employing IQCs with \emph{dynamic} multipliers instead of the static ones used in Lemma~\ref{DSP::lem::fb_s_proc} and Corollary \ref{DSP::coro::dynamic_rep_static}.

	%-----------------------------------------------------------------------------------------
	\section{Main Result}\label{DSP::sec::main}
	%-----------------------------------------------------------------------------------------
	\newcommand{\mm}{P}
	
	Our main robustness result for the interconnection \eqref{DSP::eq::interconn} is based on the IQC theorem in  \citep{SchVee18} involving dynamic multipliers.
	%Such IQC results apply for vastly more complex uncertainties than those described by \eqref{DSP::eq::uncertainty}} \emph{in case} one is able to provide IQCs capturing the behavior of those uncertainties.
	To this end, we construct such multipliers described as $\Pi = \Psi^\ast\mm \Psi$ with a fixed stable dynamic outer factor $\Psi \in \rhi^{m\times (k+l)}$ and a real middle matrix $\mm \in \S^m$ which serves as a variable and is subject to suitable constraints.
	Moreover, let us recall the following notion introduced by \cite{SchVee18}, which involves some state-space description of the filter $\Psi$. In fact, the output $y$ of $\Psi$ in response to the input $u\in L_{2e}^{k+l}$ is supposed to be given by
	\begin{equation}
		\mat{c}{\dot \xi(t) \\ y(t)} = \mat{cc}{A_\Psi & B_\Psi \\ C_\Psi & D_\Psi} \mat{c}{\xi(t) \\ u(t)}
		\label{DSP::eq::filter}
	\end{equation}
	for $t\geq 0$ and with a zero initial condition, i.e., $\xi(0) = 0$.
	
	\begin{defi}\label{def4}%-----------------------------------------------------------------------------
		The uncertainty $\Del: L_{2e}^{k} \to L_{2e}^l$ satisfies a finite-horizon IQC with terminal cost matrix $Z = Z^T$ for the multiplier $\Pi = \Psi^\ast \mm \Psi$ if the dissipation inequality
		\begin{equation*}
			\int_0^T y(t)^T \mm y(t) dt + \xi(T)^T Z \xi(T) \geq 0 \teq{ for all }T \geq 0
		\end{equation*}
		holds for the trajectories of the filter \eqref{DSP::eq::filter} driven by the input $u = \smat{z \\ \Del(z)}$ with any $z \in L_{2e}^k$.
	\end{defi}%-------------------------------------------------------------------------------

	%Note that \red{the choice $Z = 0$ recovers} a so-called hard (finite-horizon) IQC as introduced, e.g., by \cite{MegRan97}.
	
	If $\Del\in \Delf(\Vb)$ satisfies a finite-horizon IQC with terminal cost $Z$, then standard arguments involving Parseval's theorem lead to the frequency domain inequality (FDI)
	\begin{equation*}
		\mat{c}{I \\ \Del(\io)}^\ast \Pi(\io) \mat{c}{I \\ \Del(\io)} \cge 0
		\teq{ for all }\om \in \R
	\end{equation*}
	since $\Psi$ and $\Del$ are assumed to be stable.
	Note that this is the frequency dependent analogue of \eqref{DSP::lem::eq::fb_s_proc_lmic}.
	A detailed analysis of links with the classical IQCs in \citep{MegRan97} (and the corresponding stability tests) can be found in \citep{SchVee18,Sch21}.
	In particular, this notion only involves a partial storage function defined by the filter's state and $Z$ and does not rely on any information about “internal properties” of $\Delta$.

	Now we address how to systematically construct finite-horizon IQCs with a nontrivial terminal cost for $\Delf(\Vb)$ and the value set \eqref{DSP::eq::concrete_value_set} inspired by (the proof of) Corollary~\ref{DSP::coro::dynamic_rep_static}.
	To this end, we pick a dynamic MIMO scaling $\psi^\ast M \psi$ with $\psi\in\rhi^{m_\psi\times k}$ and $M\in\S^{m_\psi}$ satisfying
	\begin{equation}
		\psi(\io)^\ast M \psi(\io) \cg 0
		\teq{ for all }\om \in \R \cup \{\infty \}.
		\label{DSP::eq::scaling}
	\end{equation}
	Then the concrete choice of $\Vb$ in \eqref{DSP::eq::concrete_value_set} naturally implies
	%\begin{equation*}
	%   \mat{c}{1 \\ \delta(\io)}^{\!\ast}\!\! P_0\! \mat{c}{1 \\ \delta(\io)} \kron \psi(\io)^\ast M \psi(\io) \cge 0
	%    %
	%    \text{ ~for all~ }\om \in \R
	%\end{equation*}
	\begin{equation*}
		\mat{c}{1 \\ \del}^{\!\ast}\!\! P_0\! \mat{c}{1 \\ \del}\,\kron\,  \psi^\ast M\psi  \cge 0
		\teq{on}i\R
	\end{equation*}
	for any $\Del =\del I_k\in \Delf(\Vb)$.
	Hence, the identity
	\begin{subequations}
		\label{DSP::eq::motivating_identity}
		\begin{equation}
			\mat{c}{1 \\ \del} \kron \psi
			=\mat{c}{\psi \\ \del \psi}
			=\mat{c}{\psi \\ \psi \del}
			= \mat{cc}{\psi & 0 \\ 0 & \psi }\mat{c}{I \\ \del I_k}
		\end{equation}
		on the imaginary axis 
		yields the inequality
		\begin{equation}\label{cws1}
			\begin{aligned}
				0 &\cle \mat{c}{1 \\ \del}^{\!\ast}\!\! P_0\! \mat{c}{1 \\ \del}\kron  \psi^\ast M\psi
				= (\bullet)^\ast (P_0 \!\kron\! M) \left(\mat{c}{1 \\ \del}\! \kron\! \psi\! \right)\\
				&= (\bullet)^\ast \!\underbrace{\mat{cc}{\psi \!& 0 \\ 0 & \psi}^{\!\ast}\!\!\big(P_0\! \kron \!M \big)\! \mat{cc}{\psi \!& 0 \\ 0 & \psi}\!}_{=: \Pi}\!\mat{c}{I \\ \delta I_k}
				= (\bullet)^\ast \Pi \mat{c}{I \\ \Del}
			\end{aligned}
		\end{equation}
	\end{subequations}
	on $i\R$ for any $\Del =\del I_k\in \Delf(\Vb)$.
	Observe that this a typical S-procedure argument, but now much less commonly applied in the context of frequency domain inequalities. In particular, the dynamic scaling $\psi^\ast M \psi$ plays the role of the Lagrange multiplier appearing in such arguments.
	
	%	This observation motivates the main technical result of this paper, whose proof is given in the appendix.
	%
	%	It involves a solution of the LMI \eqref{DSP::theo::eq::positivity} that results from applying the KYP lemma to the FDI \eqref{DSP::eq::scaling}.

	This brings us to the main technical result of this paper, a lossless time-domain formulation of the FDI \eqref{cws1}, which is proved (in the appendix) by dissipativity arguments only.

	\begin{theo}%-----------------------------------------------------------------------------
		\label{DSP::theo::finite_horizon_iqc}
		Suppose that $\psi\in\rhi^{m_\psi\times k}$ has the realization $(A_\psi, B_\psi, C_\psi, D_\psi)$ where $A_\psi \in \R^{n_\psi \times n_\psi}$ is Hurwitz. If
		\begin{equation}
			\label{DSP::theo::eq::positivity}
			(\bullet)^T\mat{cc}{0 & Y \\ Y & 0} \mat{cc}{I & 0 \\ A_\psi & B_\psi} + (\bullet)^T M \mat{cc}{C_\psi & D_\psi} \cg 0
		\end{equation}
		holds for $Y\in \S^{n_\psi}$ and $M\in \S^{m_\psi}$, then any $\Del \in \Delf(\Vb)$ satisfies a finite-horizon IQC with terminal cost matrix $P_0 \kron Y$ for $\Pi = \Psi^\ast(P_0 \kron M)\Psi$ where $\Psi := \diag(\psi, \psi)$.
	\end{theo}%-------------------------------------------------------------------------------
	
	Note that the LMI \eqref{DSP::theo::eq::positivity} results from applying the KYP lemma to the FDI \eqref{DSP::eq::scaling}. Moreover, $P_0 \kron Y$ encodes how the information about the uncertainty is affecting the IQC.

	We stress that this result can also be found in \citep{Sch21}. However, there the proof heavily
	relies on \emph{frequency domain} arguments and, in particular, on the following commutation property
	for all $\om\in\R\cup\{\infty\}$:
	\begin{equation}
		\label{DSP::eq::commutation}
		\psi(\io) \Del(\io) = \psi(\io) \del(\io) =  \del(\io)\psi(\io).
	\end{equation}
	We only make use of this property in \eqref{DSP::eq::motivating_identity} for the purpose of motivation.
	In contrast and as a major benefit, the new proof in this paper only employs \emph{time domain} arguments. This means that it can be extended% is amenable to generalizations
	, e.g., to hybrid systems where no suitable notion of a frequency domain exists 
	\citep{HolSch19b}.
	Naturally, the proof evolves around a time domain analogue of \eqref{DSP::eq::commutation} that is explicitly stated in Lemma~\ref{DSP::lem::commutation} and relies on Fubini's theorem.

	\vspace{1ex}

	With Theorem \ref{DSP::theo::finite_horizon_iqc} at hand, we can now formulate our main robust stability test for the interconnection \eqref{DSP::eq::interconn}. This is a consequence of the general IQC robust stability result as stated in Theorem 4 of \cite{SchVee18}.

	\begin{theo}%-----------------------------------------------------------------------------
		\label{DSP::theo::single_repeated}
		Let $\smat{A_\psi & B_\psi \\ C_\psi & D_\psi} \in \R^{(n_\psi+m_\psi) \times (n_\psi + k)}$ with stable $A_\psi$ be given and abbreviate $J_\Psi := \diag(J_\psi, J_\psi)$ for $J \in \{A, B, C, D\}$. Then the interconnection \eqref{DSP::eq::interconn} with $\Vb$ as in \eqref{DSP::eq::concrete_value_set} is robustly stable if there exist $X\in \S^{2n_\psi+n}$, $Y\in \S^{n_\psi}$ and $M \in \S^{m_\psi}$ satisfying
		\eqref{DSP::theo::eq::positivity},
		\begin{subequations}
			\begin{equation}
				(\bullet)^T \mat{cc|c}{0 & X & \\ X & 0 & \\ \hline && P_0 \kron M} \mat{cc|c}{I& 0 & 0 \\ 0& I & 0 \\ A_\Psi & B_\Psi \smat{C \\ 0} & B_\Psi \smat{D \\ I_k} \\ 0 & A & B \\ \hline C_\Psi & D_\Psi \smat{C \\ 0} & D_\Psi \smat{D \\ I_k}} \cl 0
				\label{DSP::theo::eq::main}
			\end{equation}
			and
			\begin{equation}
				X - \mat{cc}{P_0 \kron Y & 0 \\ 0& 0} \cg 0.
				\label{DSP::theo::eq::pos}
			\end{equation}
		\end{subequations}
	\end{theo}%-------------------------------------------------------------------------------

	\vspace{1ex}
	
	First note that we recover Corollary \ref{DSP::coro::dynamic_rep_static} by choosing $n_\psi = 0$ and $D_\psi = I_k$, i.e., by restricting the filter \eqref{DSP::eq::filter} to be static and of the form $\Psi = I_{k+l}$. We also recover the analysis results from \cite{SchKoe12} with $P_0 = \smat{1 & 0 \\ 0& -1}$; then $\Delf(\Vb)$ is the set of repeated dynamic uncertainties whose $H_\infty$-norm is bounded by one. Further restricting $Y$ to $Y = 0$ and replacing \eqref{DSP::theo::eq::positivity} with $M\cg 0$ leads back to the stability criteria proposed by \cite{Bal02}. Since the latter test relies on a vanishing terminal cost, it is typically more conservative as seen in \citep{Sch21}.
	
	Let us recall the benefit of the time-domain formulation in 
	Theorem~\ref{DSP::theo::single_repeated} from \citep{FetSch18}. In fact, the underlying dissipativity proof even permits us to conclude
	\begin{equation*}
		\arraycolsep=2pt
		\mat{c}{\xi(T) \\ x(T)}^{\!T}\!\! \left(\!X \!-\! \mat{cc}{P_0 \kron Y & 0 \\ 0 & 0} \!\right)\!\mat{c}{\xi(T) \\ x(T)} \!\leq\! \mat{c}{0 \\ x(0)}^{\!T}\!\! X\! \mat{c}{0 \\ x(0)}
	\end{equation*}
	for all $T\geq 0$ and along any trajectory of the filter \eqref{DSP::eq::filter} driven by any response $u = \smat{z \\w}$ of the uncertain interconnection \eqref{DSP::eq::interconn}. % for any initial condition $x(0)\in \R^n$ and any $\Del \in \Delf(\Vb)$.
	This allows us to define ellipsoids that are robustly invariant sets for all system trajectories, which paves the way for regional robust stability analysis.
	%This paves the way for regional robust stability analysis.
	
	For more comments and an extension to uncertain interconnections with a performance channel, we refer to \cite{Sch21}. We emphasize at this point that Theorem~\ref{DSP::theo::single_repeated} provides an exact test, i.e., it is necessary and sufficient, if, roughly speaking, the filter $\psi$ is chosen such that the set $\{\psi^\ast M \psi: M = M^T\}$ constitutes a sufficiently large subspace of $\rli^{k\times k}$ \citep{ChoTit99}.

	%-----------------------------------------------------------------------------------------
	\section{Variations}\label{DSP::sec::variations}
	%-----------------------------------------------------------------------------------------
	
	The flexibility of the approach leading to Theorem \ref{DSP::theo::finite_horizon_iqc} and Theorem \ref{DSP::theo::single_repeated} permits us to adequately treat other types of dynamic uncertainties with minor adjustments of the involved arguments only. This is illustrated next.

	%-----------------------------------------------------------------------------------------
	\subsection{Dynamic Full-Block Uncertainties}
	%-----------------------------------------------------------------------------------------
	
	Suppose that the interconnection \eqref{DSP::eq::interconn} involves dynamic full-block uncertainties, i.e., uncertainties in $\Delf(\Vb)$ with
	\begin{equation}
		\label{DSP::eq::concrete_value_set2}
		\Vb := \left\{V\in \C^{l \times k} ~\middle|~ \mat{c}{I_k \\ V}^\ast P_0 \mat{c}{I_k \\ V} \cge 0 \right\}
	\end{equation}
	for some given matrix $P_0 = \smat{Q & S \\S^T & R}\in \S^{k+l}$ with ${R \cle 0}$.
	%
	%Then we obtain analogous robust stability conditions for such interconnections
	%
	In order to construct a corresponding finite-horizon IQC with terminal cost, this time, we pick a dynamic SISO scaling $\psi^\ast M \psi$ with $\psi\in\rhi^{m_\psi\times 1}$ and $M\in\S^{m_\psi}$ satisfying $\psi(\io)^\ast M \psi(\io) > 0$ for all $\om\in\R\cup\{\infty\}$.
	Then clearly
	\begin{equation*}
		\mat{c}{I_k \\ \Del}^\ast P_0 \mat{c}{I_k \\ \Del} \,\kron\,  \psi^\ast M\psi  \cge 0
		\teq{on}i\R
	\end{equation*}
	holds for any $\Del \in \Delf(\Vb)$, since the value set $\Vb$ is taken as in \eqref{DSP::eq::concrete_value_set2}.
	Next, we proceed analogously as in \eqref{DSP::eq::motivating_identity} and observe that the identity
	\begin{multline*}
		\mat{c}{I_k \\ \Del} \kron \psi
		=\mat{c}{I_k \kron \psi \\ \Del \kron \psi}
		=\mat{c}{I_k \kron \psi \\ (I_l \kron \psi) (\Del \kron 1)} \\
		= \mat{cc}{I_k \kron \psi & 0 \\ 0 & I_l \kron \psi }\mat{c}{I_k \\ \Del}
	\end{multline*}
	on the imaginary axis yields the inequality
	\begin{equation*}
		\begin{aligned}
			0 &\cle \mat{c}{I_k \\ \Del}^{\!\ast}\!\! P_0\! \mat{c}{I_k \\ \Del}\kron  \psi^\ast M\psi
			= (\bullet)^\ast (P_0 \!\kron\! M) \left(\mat{c}{I_k \\ \Del}\! \kron\! \psi\! \right)\\
			&= (\bullet)^\ast \!\underbrace{\mat{cc}{I_k \kron \psi \!& 0 \\ 0 & I_l \kron \psi}^{\!\ast}\!\!\big(P_0\! \kron \!M \big)\! \mat{cc}{I_k \kron \psi \!& 0 \\ 0 & I_l \kron \psi}\!}_{=: \Pi}\!\mat{c}{I_k \\ \Del}
		\end{aligned}
	\end{equation*}
	on $i\R$ for any $\Del \in \Delf(\Vb)$.
	This constitutes again an S-procedure argument and motivates the following robust stability test for the interconnection \eqref{DSP::eq::interconn}.

	\begin{theo}%-----------------------------------------------------------------------------
		Let $\smat{A_\psi & B_\psi \\ C_\psi & D_\psi} \in \R^{(n_\psi+m_\psi) \times (n_\psi + 1)}$ with stable $A_\psi$ be given and abbreviate
		$J_\Psi := \diag(I_k \kron J_\psi, I_l \kron J_\psi)$ for $J \in \{A, B, C, D\}$. Then the interconnection \eqref{DSP::eq::interconn} with $\Vb$ as in \eqref{DSP::eq::concrete_value_set2} is robustly stable if there exist $X\in \S^{(k+l)n_\psi + n}$, $Y\in \S^{n_\psi}$ and $M\in \S^{m_\psi}$ satisfying
		\eqref{DSP::theo::eq::positivity}, \eqref{DSP::theo::eq::main} and \eqref{DSP::theo::eq::pos}.
	\end{theo}%-------------------------------------------------------------------------------

	%-----------------------------------------------------------------------------------------
	\subsection{Dynamic Repeated Uncertainties in Intersections of Disks and Half-Planes}
	%-----------------------------------------------------------------------------------------

	Next, let us suppose again that the interconnection \eqref{DSP::eq::interconn} involves a dynamic repeated uncertainty, but now with its Nyquist plot being contained in the value set
	\begin{equation}
		\label{DSP::eq::concrete_value_set3}
		\Vb  := \left\{v I_k ~\middle|~ v \!\in\! \C \text{, } (\bullet)^\ast P_i\! \mat{c}{1 \\ v}\! \geq\! 0
		\text{ for }i = 1, \dots, \nu \right\}
	\end{equation}
	for given matrices $P_1, \dots, P_\nu \in \S^2$ with nonpositive $(2,2)$ entries.
	Note that $\Vb$ is an intersection of discs and half-planes in $\C$, which permits us to capture joint gain- and phase constraints on the uncertainties $\Delta = \delta I_k \in\Delf(\Vb)$.

	A corresponding finite-horizon IQC with terminal cost is motivated by the following observation. If recalling \eqref{DSP::eq::motivating_identity} for each summand, we infer that
	\begin{multline*}
		\arraycolsep=1pt
		0 \cle \sum_{i = 1}^\nu \mat{c}{1 \\ \!\del\!}^{\!\ast}\! P_i \!\mat{c}{1 \\ \!\del\!}\kron\psi^\ast \!M_i \psi\!\\
		=  \mat{c}{I_k \\ \Del}^\ast\!\underbrace{\mat{cc}{\psi & 0 \\ 0 & \psi}^\ast \left(\sum_{i = 1}^\nu P_i \kron M_i\right)\! \mat{cc}{\psi & 0 \\ 0 & \psi}\!}_{=: \Pi} \mat{c}{I_k \\ \Del}
	\end{multline*}
	holds on the imaginary axis for any $\Del =\del I_k\in \Delf(\Vb)$ and all
	dynamic scalings $\psi^\ast M_1 \psi, \dots, \psi^\ast M_\nu \psi$
	which are positive definite on $i\R \cup \{\infty \}$.
	Therefore, by repeating the arguments leading to Theorem \ref{DSP::theo::single_repeated} almost verbatim, we arrive at the following robustness result.
	
	\begin{theo}%-----------------------------------------------------------------------------
		\label{DSP::theo::single_repeated2}
		Let $J_\psi$ and $J_\Psi$ for $J \in \{A, B, C, D\}$ be given as in Theorem~\ref{DSP::theo::single_repeated}. Then the interconnection \eqref{DSP::eq::interconn} with $\Vb$ as in \eqref{DSP::eq::concrete_value_set3} is robustly stable if there exist $X\in \S^{2n_\psi+n}$, $Y_1, \dots, Y_\nu\in \S^{n_\psi}$ and $M_1, \dots, M_\nu \in \S^{m_\psi}$ satisfying
		\begin{equation*}
			(\bullet)^T\mat{cc}{0 & Y_i \\ Y_i & 0} \mat{cc}{I & 0 \\ A_\psi & B_\psi} + (\bullet)^T M_i \mat{cc}{C_\psi & D_\psi} \cg 0%
		\end{equation*}
		for all $i = 1, \dots, \nu$ as well as \eqref{DSP::theo::eq::main} and \eqref{DSP::theo::eq::pos} with $ P_0 \kron M$ and $ P_0 \kron Y$ respectively replaced by
		\begin{equation*}
			\sum_{i = 1}^\nu P_i \kron M_i
			\teq{ and }
			\sum_{i = 1}^\nu P_i \kron Y_i.
		\end{equation*}
	\end{theo}%-------------------------------------------------------------------------------

	Note that we make use of dynamical scalings with identical filters $\psi$. It is not too difficult to work with different filters $\psi_1, \dots, \psi_\nu$, but then one will in general face a large Lyapunov certificate $X\in \S^d$ with $d = 2\cdot \sum_{i = 1}^\nu n_{\psi_i} + n$.

	%-----------------------------------------------------------------------------------------
	\subsection{Dynamic Repeated Uncertainties in LMI Regions}
	%-----------------------------------------------------------------------------------------
	
	Next, let us assume that the interconnection \eqref{DSP::eq::interconn} involves a dynamic repeated uncertainty with a Nyquist plot located in an LMI region. Hence, we suppose that
	\begin{equation}
		\label{DSP::eq::concrete_value_set4}
		\Vb  := \left\{v I_k ~\middle|~ v \in \C \text{ and } \mat{c}{I_\nu \\ vI_\nu}^\ast P_0 \mat{c}{I_\nu \\ vI_\nu} \cge 0\right\}
	\end{equation}
	for a given matrix $P_0 = \smat{Q & S \\ S^T & R}\in \S^{2\nu}$ satisfying $R \cle 0$. For example in \citep{ChiGah96} and \citep{PeaArz00}, it is shown that LMI regions can describe a wide range of practically relevant complex convex sets that are symmetric with respect to the real axis. This offers even more flexibility in imposing gain- and phase constraints on $\Delta\in\Delf(\Vb)$ if compared to \eqref{DSP::eq::concrete_value_set3}. Observe that 
	\eqref{DSP::eq::concrete_value_set3} is recovered by 
	enforcing $Q$, $S$ and $R$ in \eqref{DSP::eq::concrete_value_set4} to be diagonal.
	
	%In particular, we recover \eqref{DSP::eq::concrete_value_set3} by enforcing $Q$, $S$ and $R$ to be diagonal.
	
	For didactic reasons, let us begin with criteria involving static multipliers. Note that, to the best of our knowledge, even these have not appeared in the literature so far.
	
	\begin{theo}%-----------------------------------------------------------------------------
		\label{DSP::theo::single_repeated3_static}
		The interconnection \eqref{DSP::eq::interconn} with $\Vb$ as in
		\eqref{DSP::eq::concrete_value_set4} is robustly stable if there exist $X \in \S^n$ and $M\in \S^{\nu k}$ satisfying $M \cg 0$,  \eqref{DSP::lem::eq::fb_s_proc_lmia} and \eqref{DSP::lem::eq::fb_s_proc_lmib} with $P$ replaced by %the multiplier
		\begin{equation*}
			P(P_0, M) := (\bullet)^T (P_0 \kron M) (I_2 \kron \vect(I_\nu) \kron I_k).
		\end{equation*}
		Here, $\vect(S)$ denotes the vectorization of the matrix $S$.
	\end{theo}%-----------------------------------------------------------------------------
	
	\begin{proof}
		Note at first that we have, by the rules of the Kronecker product and for any $V = vI_k \in \Vb$,
		\begin{equation}
			\label{DSP::kronecker_inequality}
			0 \cle \!\mat{c}{I_\nu \\ v I_\nu}^{\!\ast}\! P_0\! \mat{c}{I_\nu \\ v I_\nu} \kron M
			=  (\bullet)^\ast (P_0 \kron M) \left(\!\mat{c}{I_\nu \\ v I_\nu} \kron I_{k\nu} \!\right).
		\end{equation}
		Next, observe that
		\begin{equation*}
			\mat{c}{I_\nu \\ v I_\nu} \kron I_{k\nu} = \mat{c}{I_{k\nu^2} \\ v I_{k\nu^2}}
			= \mat{c}{I_{\nu^2} \\ v I_{\nu^2}} \kron I_k
		\end{equation*}
		and that, hence, a multiplication with $\vect(I_\nu) \kron I_k$ yields
		\begin{equation*}
			\begin{aligned}
				&\phantom{=}\,\left(\mat{c}{I_\nu \\ v I_\nu} \kron I_{k\nu} \right) \cdot (\vect(I_\nu) \kron I_k)
				= \mat{c}{\vect(I_\nu) \\ v\, \vect(I_\nu)} \kron I_k \\
				&= \left( \!\mat{cc}{\vect(I_\nu) & 0 \\ 0 & \vect(I_\nu)}\mat{c}{1 \\ v} \!\right) \kron I_k \\
				&= \left(\!\mat{cc}{\vect(I_\nu) & 0 \\ 0 & \vect(I_\nu)} \kron I_k \right) \left(\!\mat{c}{1 \\ v} \kron I_k\right)\\
				&= (I_2 \kron \vect(I_\nu) \kron I_k) \mat{c}{I_k \\ V}.
			\end{aligned}
		\end{equation*}
		By combining the latter identity  with \eqref{DSP::kronecker_inequality}, we conclude
		\begin{equation*}
			\begin{aligned}
				0 &\cle (\bullet)^\ast (P_0 \kron M) \left(\mat{c}{I_\nu \\ v I_\nu} \kron I_{k\nu} \right) \cdot (\vect(I_\nu) \kron I_k) \\
				&= (\bullet)^\ast (P_0 \kron M) (I_2 \kron \vect(I_\nu) \kron I_k) \mat{c}{I_k \\ V}\\
				&= P(P_0, M) \mat{c}{I_k \\ V}
			\end{aligned}
		\end{equation*}
		for any $V = vI_k \in \Vb$, which yields the claim.
	\end{proof}

	With this stability result at hand, we can proceed by formulating another novel robust stability result involving dynamic multipliers.
	This time, the corresponding finite-horizon IQC is motivated by considering the FDI
	\begin{equation*}
		\begin{aligned}
			0 &\cle (\bullet)^T\left[(\bullet)^\ast P_0 \mat{c}{I_\nu \\ \del I_\nu} \kron (\bullet)^\ast M (I_\nu \kron \psi) \right] (\vect(I_\nu) \kron I_k)
			\\
			&= (\bullet)^\ast \underbrace{\mat{cc}{\psi & 0 \\ 0 & \psi}^\ast P(P_0, M) \mat{cc}{\psi & 0 \\ 0 & \psi}}_{=: \Pi}\mat{c}{I_k \\ \Del}\text{\ \ on \ }i\R.
		\end{aligned}
	\end{equation*}
	With identical arguments as for the static case, this inequality holds for all $\Del = \del I_k \in \Delf(\Vb)$ and all (structured)
	dynamic scalings $(I_\nu \kron \psi)^\ast M (I_\nu \kron \psi)$ with $\psi \in \rhi^{m_\psi \times k}$ and $M \in \S^{\nu m_\psi}$ which are positive definite on $i\R \cup \{\infty \}$.

	Moreover, following the line of reasoning which leads to Theorem~\ref{DSP::theo::single_repeated} yields the following result.

	\begin{theo}%-----------------------------------------------------------------------------
		\label{DSP::theo::single_repeated4}
		Let $J_\psi$ and $J_\Psi$ for $J \in \{A, B, C, D\}$ be given as in Theorem~\ref{DSP::theo::single_repeated}. Then the interconnection \eqref{DSP::eq::interconn} with $\Vb$ as in \eqref{DSP::eq::concrete_value_set4} is robustly stable if there exist $X\in \S^{2n_\psi+n}$, $Y\in \S^{\nu n_\psi}$ and $M \in \S^{\nu m_\psi}$ satisfying
		\begin{equation*}
			(\bullet)^T\mat{cc|c}{0 & Y &\\ Y & 0 & \\ \hline  && M} \mat{cc}{I_\nu \kron I_{n_\psi} & 0 \\ I_\nu \kron A_\psi & I_\nu \kron B_\psi \\ \hline I_\nu \kron C_\psi & I_\nu  \kron D_\psi} \cg 0,
		\end{equation*}
		\eqref{DSP::theo::eq::main}, and \eqref{DSP::theo::eq::pos} with $P_0 \kron M$ and $ P_0 \kron Y$ replaced by
		\begin{equation*}
			P(P_0, M)
			\teq{ and }
			P(P_0, Y),
		\end{equation*}
		respectively.
	\end{theo}%-------------------------------------------------------------------------------
	
	Note that if the matrices $Q, S, R$ characterizing the set \eqref{DSP::eq::concrete_value_set4} are diagonal, then we recover Theorem \ref{DSP::theo::single_repeated2} by restricting $Y$ and $M$ to be block diagonal matrices.

	%-----------------------------------------------------------------------------------------
	\subsection{Equation Constraints}
	%-----------------------------------------------------------------------------------------

	Finally, let us suppose that the interconnection \eqref{DSP::eq::interconn} involves uncertainties in $\Delf(\Vb)$ with
	\begin{equation}
		\label{DSP::eq::concrete_value_set5}
		\Vb := \left\{vI_k ~\middle|~ (\bullet)^\ast P_0 \mat{c}{1 \\ v} \cge 0 \text{ and }(\bullet)^\ast \mat{cc}{0 & i \\ -i & 0}\mat{c}{1 \\ v} = 0\right\}
	\end{equation}
	for some given matrix $P_0 = \smat{q & s \\s & r}\in \S^{2}$ with $r \cle 0$. Note that the equation constraint restricts the scalar $v$ to be real, which implies that $\Delf(\Vb)$ only contains parametric uncertainties.
	
	In order to construct a suitable finite-horizon IQC with terminal cost, we pick again a dynamic scaling $\psi^\ast M \psi$ with $\psi\in\rhi^{m_\psi\times k}$ and $M\in\S^{m_\psi}$ satisfying \eqref{DSP::eq::scaling} for the inequality in \eqref{DSP::eq::concrete_value_set5} and another dynamic scaling $\psi^\ast N \psi$ with the same filter $\psi$ and a matrix $N \in \R^{m_\psi \times m_\psi}$ satisfying
	\begin{equation*}
		0 = (\psi^\ast N \psi )^\ast + \psi^\ast N \psi = \psi^\ast (N^T + N)\psi
		\text{ ~on~ }i\R\cup \{\infty\}
	\end{equation*}
	for the equation constraint.
	The particular choice for $\psi^\ast N \psi$ permits us to conclude
	\begin{equation*}
		\begin{aligned}
			0 &= \mat{c}{1 \\ \del}^\ast \mat{cc}{0 & i \\ -i & 0}\mat{c}{1 \\ \del} \kron (-i\cdot \psi^\ast N \psi) \\
			&= (\bullet)^\ast \left(\mat{cc}{0 & i \\ -i & 0} \kron (-i\cdot N) \right) \left( \mat{c}{1 \\ \del} \kron \psi\right)\\
			&=\! (\bullet)^\ast\! \mat{cc}{0 & N \\ -N\! & 0}\!\!\mat{cc}{\psi\! & 0 \\ 0 & \!\psi}\!\! \mat{c}{I_k \\ \Del}\!
			=\! (\bullet)^\ast\! \mat{cc}{0 & N \\ N^T & 0}\!\!\mat{cc}{\psi\! & 0 \\ 0 & \!\psi}\!\! \mat{c}{I_k \\ \Del}
		\end{aligned}
	\end{equation*}
	on the imaginary axis and for any $\Del = \del I_k \in \Delf(\Vb)$.
	Together with \eqref{DSP::eq::motivating_identity}, we get
	\begin{equation*}
		\begin{aligned}
			0 &\cle  (\bullet)^\ast \!P_0 \!\mat{c}{1 \\ \del} \kron \psi^\ast\! M\psi + (\bullet)^\ast\! \mat{cc}{0 & i \\ -i & 0}\!\mat{c}{1 \\ \del} \kron (-i \psi^\ast\! N \psi)\\
			&= (\bullet)^\ast\!\underbrace{\mat{cc}{\psi \!& 0 \\ 0 & \psi}^{\!\ast}\!\left(P_0\! \kron \!M  + \mat{cc}{0 & N \\ N^T & 0}\right)\! \mat{cc}{\psi \!& 0 \\ 0 & \psi}}_{=: \Pi}\!\mat{c}{I_k \\ \Del}
		\end{aligned}
	\end{equation*}
	on $i\R$ and for any $\Del = \del I_k \in \Delf(\Vb)$.
	Note that the choice for $\psi^\ast N \psi$ assures that $-i\cdot \psi^\ast N \psi$ is Hermitian on $i\R$ and that $\Pi \in \rli^{2k \times 2k}$ has a symmetric middle matrix.
	
	Another variation of the proof of Theorem \ref{DSP::theo::finite_horizon_iqc} then leads to the following result, which is related to $\mu$-analysis via dynamic D/G-scalings if choosing $P_0 = \smat{1 & 0 \\ 0 & -1}$.

	\begin{theo}%-----------------------------------------------------------------------------
		\label{DSP::theo::single_repeated5}
		Let $J_\psi$ and $J_\Psi$ for $J \in \{A, B, C, D\}$ be given as in Theorem~\ref{DSP::theo::single_repeated}. Then the interconnection \eqref{DSP::eq::interconn} with $\Vb$ as in \eqref{DSP::eq::concrete_value_set5} is robustly stable if there exist $X\in \S^{2n_\psi+n}$, $Y\in \S^{n_\psi}$, $Z \in \S^{n_\psi}$, $M\in \S^{m_\psi}$ and $N \in \R^{m_\psi \times m_\psi}$ satisfying \eqref{DSP::theo::eq::positivity},
		\begin{equation*}
			(\bullet)^T\mat{cc}{0 & Z \\ Z & 0} \mat{cc}{I & 0 \\ A_\psi & B_\psi} + (\bullet)^T (N^T+N) \mat{cc}{C_\psi & D_\psi} = 0%
		\end{equation*}
		as well as \eqref{DSP::theo::eq::main} and \eqref{DSP::theo::eq::pos} with $P_0 \kron M$ and $ P_0 \kron Y$ replaced by
		\begin{equation*}
			P_0 \kron M + \mat{cc}{0 & N \\ N^T & 0}
			\teq{ and }
			P_0 \kron Y + \frac{1}{2}\mat{cc}{0 & Z \\ Z & 0},
		\end{equation*}
		respectively.
	\end{theo}%-------------------------------------------------------------------------------
	
	Note that this result can be related to the generalized KYP lemma proposed by \cite{IwaHar05}, but this aspect is not further explored here.
	
	%-----------------------------------------------------------------------------------------
	\section{Example}\label{DSP::sec::example}
	%-----------------------------------------------------------------------------------------
	
	\begin{figure}
		\begin{center}
			\includegraphics[width=0.48\textwidth, trim=20 130 20 160, clip]{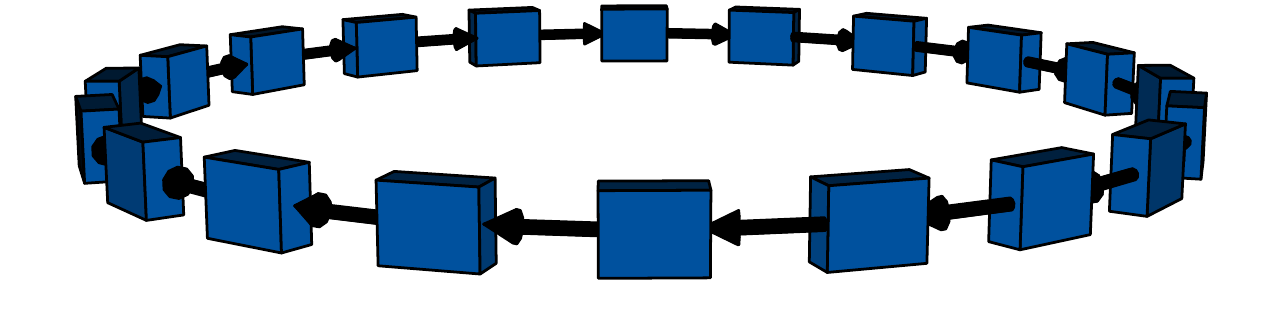}
		\end{center}
		\caption{Cyclic interconnected networked system \eqref{DSP::eq::network}.}
		\label{DSP::fig::network}
	\end{figure}

	As an illustration let us consider a networked system composed of $N = 20$ identical subsystems with description
	\begin{subequations}
		\label{DSP::eq::network}
		\begin{equation}
			\mat{c}{\dot x_k(t) \\ \hline z_k(t) \\ e_k(t)}
			= \mat{c|cc}{A & B & B_2 \\ \hline C & D & D_{12} \\ C_2 & D_{21} & D_{22}}\mat{c}{x_k(t) \\ \hline w_k(t) \\ d_k(t)}
			\label{DSP::eq::subsystem}
		\end{equation}
		for $t \geq 0$ and $k\in \{1,\dots, N\}$, where $d_1, \dots, d_N \in L_2$ are input disturbances and $e_1, \dots, e_N$ are error output signals. We suppose that the matrices describing the subsystems are given by
		\begin{equation}
			\arraycolsep=3pt
			\mat{c|c:c}{A & B & B_2 \\ \hline C & D & D_{12} \\ \hdashline C_2 & D_{21} & D_{22}}
			= \mat{cc|cc:c}{-13 & -12 & 10 & 0 & 1 \\ 1 & 0 & 0 & 0 & 0  \\ \hline -10.1 & -11.2 & 10 & 1 & 0 \\ 1 & 2 & 0 & 0 & 1 \\ \hdashline 1 & 0 & 0 & 0 & 0}
			\label{DSP::eq::subsystem_matrices}
		\end{equation}
		and that the subsystems are coupled as
		\begin{equation}
			w_k(t) = \sum_{j = 1}^N a_{kj}(z_k(t) - z_j(t))
			\label{DSP::eq::network_coupling}
		\end{equation}
		for $t\geq 0$ and $k\in \{1, \dots, N\}$. More concretely, we suppose that the link strengths $a_{kj}$ are uncertain and satisfy
		\begin{equation}
			a_{kj} \in \begin{cases}
				[0.75, 1] & \text{ if }j = k+1 \text{ or }(j,k) = (N,1) \\
				\{0\} & \text{ otherwise}
			\end{cases}
			\label{DSP::eq::network_coupling_cyclic}
		\end{equation}
	\end{subequations}
	for all $k,j \in \{1, \dots, N\}$. This means that we consider a directed cyclic interconnected network as depicted in Fig.~\ref{DSP::fig::network} with uncertain link strengths.

	For example on the basis of the results from \cite{Wie10}, it is not difficult to show that the networked system \eqref{DSP::eq::network} is robustly stable and admits a robust energy gain smaller than $\ga$ if and only if the uncertain interconnection
	\begin{equation}
		\label{DSP::eq::auxiliary_system}
		\mat{c}{\dot x(t) \\ \hline z(t) \\ e(t)}
		= \mat{c|cc}{A & B & B_2 \\ \hline C & D & D_{12} \\ C_2 & D_{21} & D_{22}}\mat{c}{x(t) \\ \hline w(t) \\ d(t)},\quad
		w(t) = \del z(t)
	\end{equation}
	is robustly stable and admits a robust energy gain smaller than $\ga$. Here, $\delta$ is a single repeated parametric uncertainty satisfying
	\begin{equation}
		\del \in \t \vb := \bigcup_{(a_{kj})_{k,j=1}^N \text{ with \eqref{DSP::eq::network_coupling_cyclic}}}\eig\left(\Ls\big((a_{kj})_{k,j=1}^N\big)\right)
		\label{DSP::eq::original_uncertainty_set}
	\end{equation}
	where $\Ls\big((a_{kj})_{k,j=1}^N\big) =: \Ls(a)$ denotes the Laplacian matrix corresponding to the graph in Fig.~\ref{DSP::fig::network} for the instance $a$ of link strengths satisfying \eqref{DSP::eq::network_coupling_cyclic}. Precisely, this matrix is given by
	\begin{equation*}
		\Ls(a)_{kj} := \begin{cases}
			\sum_{j = 1}^N a_{kj} & \text{ if }j = k, \\
			-a_{kj} & \text{ otherwise.}
		\end{cases}
	\end{equation*}
	The set $\t \vb$ in \eqref{DSP::eq::original_uncertainty_set} is somewhat complicated, but we can find a suitable superset for applying our robust analysis tools. Indeed, one can show that
	\begin{equation}
		\label{DSP::eq::covering}
		\t \vb \subset \vb := \big\{v\in \C\,:\, |v-1| \leq 1 \big\} \cap \left\{v\in \C\,:\, \left|v - \frac{3}{4}\right| \geq \frac{3}{4}\right\}.
	\end{equation}
	The boundary of the set $\vb$ and the eigenvalues of $\Ls(a)$ for several selected and randomly chosen values of
	$a$ with \eqref{DSP::eq::network_coupling_cyclic} are depicted in Fig.~\ref{DSP::fig::eigenvalues}.

	\begin{figure}
		\includegraphics[width=0.5\textwidth, trim= 10 73 10 20, clip]{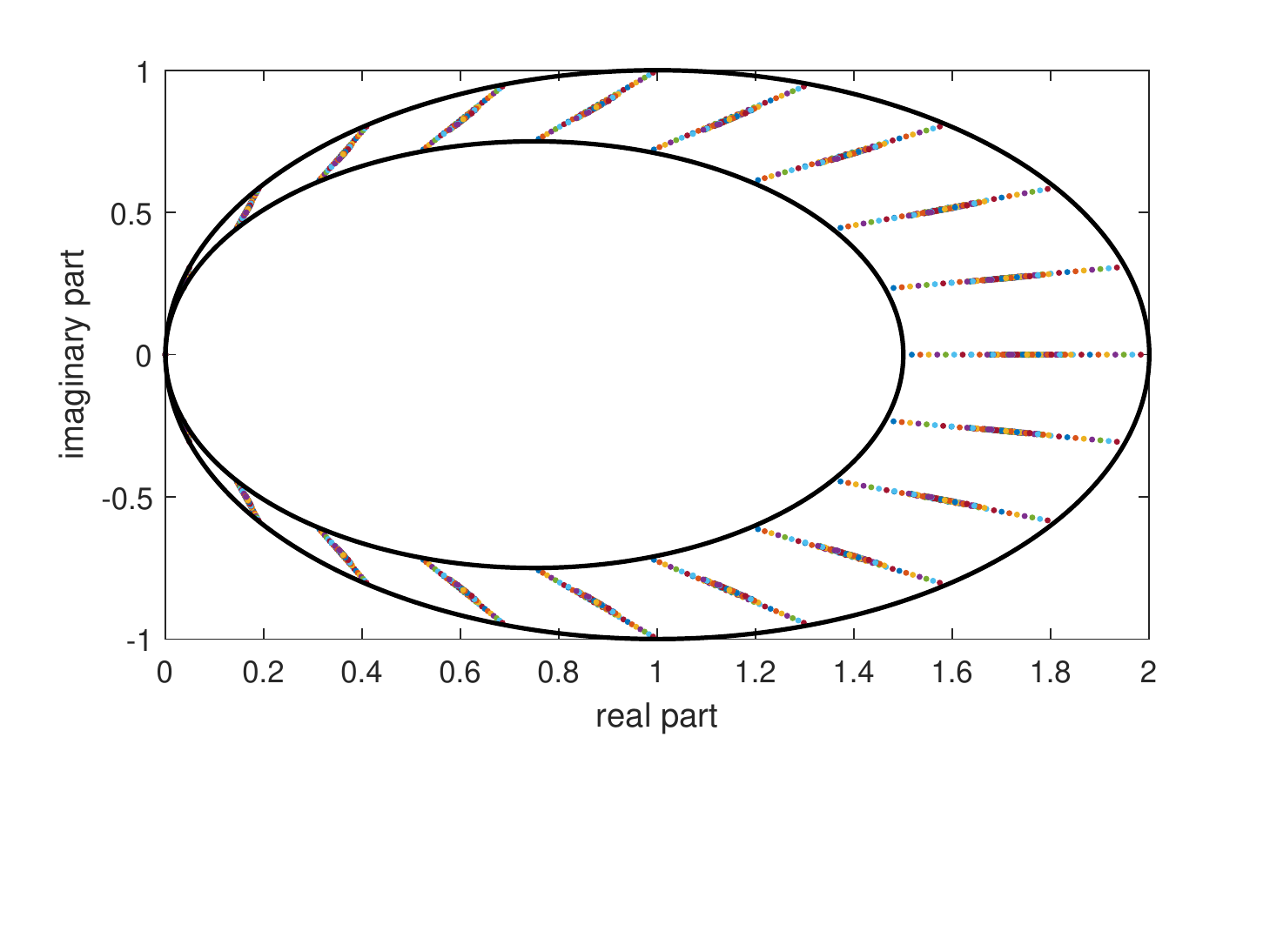}
		\caption{Eigenvalues of the Laplacian matrix $\Ls(a)$ for several instances of link strengths \eqref{DSP::eq::network_coupling_cyclic} and boundary of the set $\vb$ in \eqref{DSP::eq::covering}.}
		\label{DSP::fig::eigenvalues}
	\end{figure}
	
	Note that a robust analysis of the system \eqref{DSP::eq::auxiliary_system} based on covering $\t \vb$ with a single disk $\big\{v\in \C\,:\, |v-1| \leq 1 \big\}$ is doomed to fail, since $\del = 0.5$ is contained in this disk and since \eqref{DSP::eq::auxiliary_system} is unstable for this particular value of $\del$.
	However, since we can express the set $\Vb := \{vI_2~:~ v\in \vb \}$ equivalently as \eqref{DSP::eq::concrete_value_set3} with
	\begin{equation*}
		P_1 := \mat{cc}{0 & 1 \\ 1 & -1} \teq{ and } P_2 := \mat{cc}{0 & -0.75 \\ -0.75 & 1}
	\end{equation*}
	we can freely apply the standard extension of Theorem \ref{DSP::theo::single_repeated2}, \ref{DSP::theo::single_repeated3_static} or \ref{DSP::theo::single_repeated4} for analyzing uncertain systems with a performance channel. Recall that we do not require convexity of $\Vb$ since \eqref{DSP::eq::auxiliary_system} only involves parametric uncertainties.
	Applying the extension of Theorem \ref{DSP::theo::single_repeated2} for $n_\psi = 0$ (static scaling) and for $n_\psi = 2$, $n_\psi = 4$ (dynamic scaling) assures robust stability of the network \eqref{DSP::eq::network} and yields the optimal upper bounds
	\begin{equation*}
		0.654, \quad 0.572 \teq{ and }0.572
	\end{equation*}
	on its robust energy gain, respectively; we did choose the filter $\psi$ as in \citep{SchKoe12} with $\alpha = 2$ and $\nu = 0$, $\nu = 1$ and $\nu = 2$, respectively.
	Similarly as in \citep{PooTik95}, this demonstrates that there is a benefit of using dynamic scalings over static ones even if their McMillian degree is small.

	%-----------------------------------------------------------------------------------------
	\section{Conclusion}\label{DSP::sec::conclusions}
	%-----------------------------------------------------------------------------------------

	We provide an alternative proof of one of the results in \citep{Sch21} for analyzing systems affected by dynamic uncertainties by means of IQCs with dynamic multipliers. In contrast to \cite{Sch21}, we do not rely on frequency domain arguments.
	Moreover, we provide several interesting variations that permit us, e.g., to analyze the robustness of a system against dynamic uncertainties with a Nyquist plot known to be located in a given LMI region.
	
	%The proposed results are expected to extend to robust estimation problems without much difficulties and
	%Gain-scheduling synthese bei den einfachen sachen ist relativ klar bei den summen ist dualisieren ungeschickt (der inverse multiplikator hat kann man nicht gut mit LMIs fassen) (robust estimation sollte in allen faellen gut gehen)

	%\begin{ack}
	%Place acknowledgments here.
	%\end{ack}
	
	\bibliography{literatur}  % bib file to produce the bibliography

\begin{thebibliography}{25}
\providecommand{\natexlab}[1]{#1}
\providecommand{\url}[1]{\texttt{#1}}
\providecommand{\urlprefix}{URL }
\expandafter\ifx\csname urlstyle\endcsname\relax
  \providecommand{\doi}[1]{doi:\discretionary{}{}{}#1}\else
  \providecommand{\doi}{doi:\discretionary{}{}{}\begingroup
  \urlstyle{rm}\Url}\fi

\bibitem[{Balakrishnan(2002)}]{Bal02}
Balakrishnan, V. (2002).
\newblock {L}yapunov {F}unctional in {C}omplex $\mu$ {A}nalysis.
\newblock \emph{IEEE Trans. Autom. Control}, 47(9), 1466--1479.
\newblock \doi{10.1109/TAC.2002.802766}.

\bibitem[{Ben-Tal and Nemirovski(2001)}]{BenNem01}
Ben-Tal, A. and Nemirovski, A. (2001).
\newblock \emph{Lectures on Modern Convex Optimization: Analysis, Algorithms,
  and Engineering Applications}.
\newblock SIAM.
\newblock \doi{10.1137/1.9780898718829}.

\bibitem[{Boyd et~al.(1994)Boyd, El~Ghaoui, Feron, and Balakrishnan}]{BoyGha94}
Boyd, S., El~Ghaoui, L., Feron, E., and Balakrishnan, V. (1994).
\newblock \emph{Linear Matrix Inequalities in System \& Control Theory}.
\newblock Society for Industrial \& Applied.
\newblock \doi{10.1137/1.9781611970777}.

\bibitem[{Boyd and Vandenberghe(2004)}]{BoyVan04}
Boyd, S.P. and Vandenberghe, L. (2004).
\newblock \emph{Convex Optimization}.
\newblock Cambridge University Press.
\newblock \doi{10.1017/CBO9780511804441}.

\bibitem[{Chilali and Gahinet(1996)}]{ChiGah96}
Chilali, M. and Gahinet, P. (1996).
\newblock ${H}_\infty$ design with pole placement constraints: an {LMI}
  approach.
\newblock \emph{{IEEE} Trans. Autom. Control}, 41(3), 358--367.
\newblock \doi{10.1109/9.486637}.

\bibitem[{Chou et~al.(1999)Chou, Tits, and Balakrishnan}]{ChoTit99}
Chou, Y.S., Tits, A.L., and Balakrishnan, V. (1999).
\newblock Stability multipliers and $\mu$ upper bounds: Connections and
  implications for numerical verification of frequency domain conditions.
\newblock \emph{IEEE Trans. Autom. Control}, 44(5), 906--913.
\newblock \doi{10.1109/9.763207}.

\bibitem[{Fetzer et~al.(2018)Fetzer, Scherer, and Veenman}]{FetSch18}
Fetzer, M., Scherer, C.W., and Veenman, J. (2018).
\newblock Invariance with dynamic multipliers.
\newblock \emph{{IEEE} Trans. Autom. Control}, 63(7), 1929--1942.
\newblock \doi{10.1109/TAC.2017.2762764}.

\bibitem[{Gusev and Likhtarnikov(2006)}]{GusLik06}
Gusev, S.V. and Likhtarnikov, A.L. (2006).
\newblock {K}alman-{P}opov-{Y}akubovich lemma and the {S}-procedure: {A}
  historical essay.
\newblock \emph{Autom. Rem. Control}, 67(11), 1768--1810.
\newblock \doi{10.1134/S000511790611004X}.

\bibitem[{Holicki and Scherer(2019)}]{HolSch19b}
Holicki, T. and Scherer, C.W. (2019).
\newblock Stability analysis and output-feedback synthesis of hybrid systems
  affected by piecewise constant parameters via dynamic resetting scalings.
\newblock \emph{Nonlinear Anal. Hybri.}, 34, 179--208.
\newblock \doi{10.1016/j.nahs.2019.06.003}.

\bibitem[{Horn and Johnson(1991)}]{HorJoh91}
Horn, R.A. and Johnson, C.R. (1991).
\newblock \emph{Topics in Matrix Analysis}.
\newblock Cambridge Univ. Press.
\newblock \doi{10.1017/CBO9780511840371}.

\bibitem[{Iwasaki and Hara(2005)}]{IwaHar05}
Iwasaki, T. and Hara, S. (2005).
\newblock Generalized {KYP} lemma: unified frequency domain inequalities with
  design applications.
\newblock \emph{IEEE Trans. Autom. Control}, 50(1), 41--59.
\newblock \doi{10.1109/TAC.2004.840475}.

\bibitem[{Megretsky and Rantzer(1997)}]{MegRan97}
Megretsky, A. and Rantzer, A. (1997).
\newblock System analysis via integral quadratic constraints.
\newblock \emph{IEEE Trans. Autom. Control}, 42(6), 819--830.
\newblock \doi{10.1109/9.587335}.

\bibitem[{Packard and Doyle(1993)}]{PacDoy93}
Packard, A. and Doyle, J. (1993).
\newblock The complex structured singular value.
\newblock \emph{Automatica}, 29(1), 71--109.
\newblock \doi{10.1016/0005-1098(93)90175-S}.

\bibitem[{Peaucelle et~al.(2000)Peaucelle, Arzelier, Bachelier, and
  Bernussou}]{PeaArz00}
Peaucelle, D., Arzelier, D., Bachelier, O., and Bernussou, J. (2000).
\newblock A new robust ${D}$-stability condition for real convex polytopic
  uncertainty.
\newblock \emph{Syst. Control Lett.}, 40(1), 21--30.
\newblock \doi{10.1016/S0167-6911(99)00119-X}.

\bibitem[{P{\'{o}}lik and Terlaky(2007)}]{PolTer07}
P{\'{o}}lik, I. and Terlaky, T. (2007).
\newblock A survey of the {S}-lemma.
\newblock \emph{SIAM Rev.}, 49(3), 371--418.
\newblock \doi{10.1137/S003614450444614X}.

\bibitem[{Poolla and Tikku(1995)}]{PooTik95}
Poolla, K. and Tikku, A. (1995).
\newblock Robust performance against time-varying structured pertubations.
\newblock \emph{IEEE Trans. Autom. Control}, 40(9), 1589--1602.
\newblock \doi{10.1109/9.412628}.

\bibitem[{Rantzer(1996)}]{Ran96}
Rantzer, A. (1996).
\newblock On the {K}alman-{Y}akubovich-{P}opov lemma.
\newblock \emph{Syst. Control Lett.}, 28(1), 7--10.
\newblock \doi{10.1016/0167-6911(95)00063-1}.

\bibitem[{Scherer(1997)}]{Sch97}
Scherer, C.W. (1997).
\newblock A full block {S}-procedure with applications.
\newblock In \emph{Proc. 36th IEEE Conf. Decision and Control}, 2602--2607.
\newblock \doi{10.1109/CDC.1997.657769}.

\bibitem[{Scherer(2001)}]{Sch01}
Scherer, C.W. (2001).
\newblock {LPV} control and full block multipliers.
\newblock \emph{Automatica}, 37(3), 361--375.
\newblock \doi{10.1016/S0005-1098(00)00176-X}.

\bibitem[{Scherer(2021)}]{Sch21}
Scherer, C.W. (2021).
\newblock Dissipativity and integral quadratic constraints: Tailored
  computational robustness tests for complex interconnections.
\newblock arXiv.
\newblock \doi{10.48550/arXiv.2105.07401}.

\bibitem[{Scherer and K\"ose(2012)}]{SchKoe12}
Scherer, C.W. and K\"ose, I.E. (2012).
\newblock Gain-scheduled control synthesis using dynamic ${D}$-scales.
\newblock \emph{IEEE Trans. Autom. Control}, 57(9), 2219--2234.
\newblock \doi{10.1109/TAC.2012.2184609}.

\bibitem[{Scherer and Veenman(2018)}]{SchVee18}
Scherer, C.W. and Veenman, J. (2018).
\newblock On merging frequency-domain techniques with time domain conditions.
\newblock \emph{Syst. Control Lett.}, 121, 7--15.
\newblock \doi{10.1016/j.sysconle.2018.08.005}.

\bibitem[{Scherer and Weiland(2000)}]{SchWei00}
Scherer, C.W. and Weiland, S. (2000).
\newblock \emph{Linear Matrix Inequalities in Control}.
\newblock Lecture Notes, Dutch Inst. Syst. Control, Delft.

\bibitem[{Uhlig(1979)}]{Uhl79}
Uhlig, F. (1979).
\newblock A recurring theorem about pairs of quadratic forms and extensions: a
  survey.
\newblock \emph{Linear Alg. Appl.}, 25, 219--237.
\newblock \doi{10.1016/0024-3795(79)90020-X}.

\bibitem[{Wieland(2010)}]{Wie10}
Wieland, P. (2010).
\newblock \emph{From Static to Dynamic Couplings in Consensus and
  Synchronization among Identical and Non-Identical Systems}.
\newblock Ph.D. thesis, University of Stuttgart.
\newblock \doi{10.18419/opus-4295}.

\end{thebibliography}
	% with bibtex (preferred)

	\appendix
	
	\section{Auxiliary Results and Technical Proofs}    % Each appendix must have a short title.
	\begin{lemm}%-----------------------------------------------------------------------------
		\label{DSP::lem::commutation}
		Let $g: L_{2e}^1 \to L_{2e}^1$ and $H: L_{2e}^k \to L_{2e}^l$ be convolution maps represented as in \eqref{DSP::eq::uncertainty}. Then we have
		\begin{equation*}
			H \circ (g I_k)  = (g I_l)\circ H
		\end{equation*}
		with $g I_j \!:\! L_{2e}^j \!\to\! L_{2e}^j$, $w \!\mapsto\! (g(w_1), \dots, g(w_j))^T$ for $j\! \in\! \{k,l\}$.
	\end{lemm}%-------------------------------------------------------------------------------

	\begin{proof}
		Let $g$ and $H$ be realized by $(A_g, B_g, C_g, D_g)$ and $(A, B, C, D)$, respectively. Further, let us abbreviate the functions $\t g(s) := C_g e^{A_g s}B_g$ and $\t H(s) := C e^{A s}B$. Since $g$ is SISO, note that we have in particular $\t H \t g = \t g \t H$,
		\begin{equation}
			%, \quad
			%
			\t g D = D \t g, \quad
			\t H D_g = D_g \t H
			\text{ ~and~ }
			DD_g = D_g D.
			\label{SWD::pro::commut_siso}
		\end{equation}
		Now, let $w \in L_{2e}^k$ and $t \geq 0$ be arbitrary. Then we have via integration by substitution
		\begin{equation*}
			\begin{aligned}
				&\phantom{=}\int_{0}^t  \t H(t-s) \left(\int_{0}^s \t g(s-r)w(r) \,dr \right)\,ds\\
				&= \int_0^{t}  \hspace{-1ex}\t H(s) \left(\int_{0}^{t-s} \t g(t-s-r)w(r) \,dr \right)\,ds  \\
				&= \int_0^{t} \hspace{-1ex} \t H(s) \left(\int_0^{t-s}\t g(r) w(t-s-r)\,dr \right) \,ds  \\
				&= \int_0^{t}  \left(\int_0^{t}\t H(s) \t g(r) w(t-s-r)\chi(r, t-s)\,dr \right)\,ds;
			\end{aligned}
		\end{equation*}
		here $\chi(a,b) = 1$ if $a \leq b$ and $\chi(a,b) = 0$ otherwise.
		Using $\t H \t g = \t g \t H$ and Fubini's theorem, the last term equals
		\begin{equation*}
			\begin{aligned}
				&\phantom{=} \int_0^{t} \left(\int_0^{t}  \t g(r)\t H(s) w(t-s-r)\chi(r, t-s) \,dr \right)\,ds \\
				&= \int_0^{t}  \left(\int_0^{t} \t g(r)\t H(s) w(t-s-r)\chi(r, t-s) \,ds \right)\,dr \\
				&= \int_0^{t} \left(\int_0^{t}\t g(r) \t H(s)w(t-s-r)\chi(s, t-r) \,ds \right)\,dr.
			\end{aligned}
		\end{equation*}
		Again via integration by substitution, this is the same as
		\begin{equation*}
			\begin{aligned}		
				&\phantom{=} \int_0^{t} \t g(r)\left( \int_0^{t-r}  \t H(r) w(t-s-r) \,ds \right)\,dr  \\
				&=  \int_0^{t}  \t g(r)\left(\int_{0}^{t-r} \t H(t-r-s) w(s) \,ds \right)\,dr \\
				&=  \int_{0}^t  \t g(t-r)\left(\int_{0}^{r}  \t H(r-s) w(s) \,ds \right)\,dr
			\end{aligned}
		\end{equation*}
		which yields the statement for $D = 0$ and $D_g = 0$. The general case is obtained by using linearity and \eqref{SWD::pro::commut_siso}.
	\end{proof}

	\begin{proof}[Proof of Theorem \ref{DSP::theo::finite_horizon_iqc}]
		Let $\Del \in \Delf(\Vb)$ be arbitrary. Then, there exists some $\del \in \rhi^{1 \times 1}$ with $\Del = \del I_k$. Let $(A_\del, B_\del, C_\del, D_\del)$ be a minimal realization of $\del$ and recall that $A_\del$ is Hurwitz. Further, note that we can work with $(A_\del \kron I_k, B_\del \kron I_k, C_\del \kron I_k, D_\del \kron I_k)$ as a realization of $\Del$.
		
		Since the value set $\Vb$ is chosen as in \eqref{DSP::eq::concrete_value_set}, the FDI
		\begin{equation*}
			\mat{c}{1 \\ \del(\io)}^\ast P_0 \mat{c}{1 \\ \del(\io)} \geq 0 \teq{ holds for all } \om \in \R\cup \{\infty\}.
		\end{equation*}
		Consequently, we can infer the existence of a symmetric matrix $W$ satisfying
		\begin{equation}
			\label{DSP::pro::eq::siso_kyp_lmi}
			(\bullet)^T \mat{cc}{0 & W \\ W & 0} \mat{cc}{I & 0 \\ A_\del & B_\del} + (\bullet)^T P_0 \mat{cc}{0 & I \\ C_\del & D_\del} \cge 0
		\end{equation}
		by the the KYP lemma. Since $A_\del$ is Hurwitz, $r \leq 0$ and since the left upper block of \eqref{DSP::pro::eq::siso_kyp_lmi} is a standard Lyapunov inequality, we can conclude that $W \cle 0$ holds.
		
		Next, let $K\in\S^d$ with $d:=n_\psi +k$ denote the whole left hand side of \eqref{DSP::theo::eq::positivity} and observe that we can merge %incorporate this positive definite matrix into the LMI
		the LMIs $K\cg 0$ and \eqref{DSP::pro::eq::siso_kyp_lmi}  to obtain
		\begin{equation*}
			\left((\bullet)^T \mat{cc}{0 & W \\ W & 0} \mat{cc}{I & 0 \\ A_\del & B_\del} + (\bullet)^T P_0 \mat{cc}{0 & I \\ C_\del & D_\del} \right) \kron K \cge 0
		\end{equation*}
		by standard properties of $\otimes$. By using further rules of this product, we get the identities
		\begin{equation*}
			\begin{aligned}
				&\phantom{=}\,~ (\bullet)^T \mat{cc}{0 & W \\ W & 0} \mat{cc}{I & 0 \\ A_\del & B_\del}\kron K\\&=
				(\bullet)^T \left(\mat{cc}{0 & W \\ W & 0}\kron K\right)\left(
				\mat{cc}{I & 0 \\ A_\del & B_\del}\kron I_d\right)\\&=
				(\bullet)^T \mat{cc}{0 & W\kron K \\ W\kron K & 0}
				\mat{cc}{I & 0 \\ A_\del\kron I_d & B_\del\kron I_d}
			\end{aligned}
		\end{equation*}
		and, with analogous intermediate steps,
		\begin{equation*}
			(\bullet)^T P_0 \mat{cc}{0 & I \\ C_\del & D_\del}\kron K =
			(\bullet)^T (P_0\kron K)\mat{cc}{0 & I \\ C_\del\kron I_d & D_\del\kron I_d}.
		\end{equation*}
		These permit us to conclude
		\begin{equation}
			\label{DSP::pro::eq::siso_kyp_lmi2}
			(\bullet)^T \!\mat{cc|c}{0 & W \kron K &\\ W \kron K & 0& \\ \hline  && P_0\kron K} \!\mat{cc}{I & 0 \\ A_\del \kron I_d & B_\del \kron I_d \\ \hline 0 & I \\ C_\del \kron I_d & D_\del \kron I_d}\! \cge\! 0.
		\end{equation}
		
		As a next step toward showing the validity of the IQC in Definition~\ref{def4}, we define $u = \smat{u_1 \\ u_2} = \smat{z \\ \Del(z)}$ for some arbitrary $z\in L_{2e}^k$.
		Next to this signal, we also work with the output $y=\smat{y_1 \\ y_2}$ and the state $\xi=\smat{\xi_1 \\ \xi_2}$ of the filter \eqref{DSP::eq::filter} in response to the input $u$. Here, the partitions are induced by the block diagonal structure of the realization matrices. Finally, we define the auxiliary signals
		\begin{equation*}
			\t y_i := \mat{c}{\xi_i\\u_i} \in L_{2e}^d
			\text{\ \ for\ \ }i \in \{1, 2\}.
		\end{equation*}
		By the definition of $K\in\S^d$ as the left-hand side of \eqref{DSP::theo::eq::positivity},
		we infer for all $t\geq 0$ and $i,j\in\{1,2\}$ that
		$$
		y_i(t)^TM y_j(t)+\frac{d}{dt}\xi_i(t)^T Y \xi_j(t)=\t y_i(t)^TK\t y_j(t)
		$$
		holds. As a consequence of direct calculations, we get
		\begin{equation}
			(\bullet)^T\!\big(P_0 \kron M \big) y(t) + \frac{d}{dt}(\bullet)^T\! \big( P_0 \kron Y\big) \xi(t)
			= (\bullet)^T\!\big(P_0 \kron K \big) \t y(t)
			\label{DSP::eq::K_signal_expression}
		\end{equation}
		for all $t\geq 0$.
		
		Now recall that $u_2=\Delta(u)=(\delta I_k)(u_1)$ holds and that $\t y_i$ is the response of an LTI filter with  zero initial condition to the input $u_i$ for $i\in\{1,2\}$. Therefore, Lemma~\ref{DSP::lem::commutation} implies
		\begin{equation*}
			\t y_2 = (\del I_d)(\t y_1).
		\end{equation*}
		Note that we can equivalently express this identity as
		\begin{equation*}
			\mat{c}{\dot \zeta(t) \\ \t y_2(t)} = \mat{cc}{A_\del \kron I_d& B_\del \kron I_d\\ C_\del \kron I_d & D_\del \kron I_d} \mat{c}{\zeta(t) \\ \t y_1(t)}
		\end{equation*}
		for $t\geq 0$ and with $\zeta(0)=0$.
		
		At this point we exploit \eqref{DSP::pro::eq::siso_kyp_lmi2}. By standard dissipation arguments, the function $\eta: t \mapsto \zeta(t)^T(W \kron K)\zeta(t)$ satisfies
		\begin{equation*}
			\begin{aligned}
				\dot \eta(t) &= (\bullet)^T\! \mat{cc}{0 & W\! \kron\! K \\ W\! \kron\! K & 0} \!\mat{cc}{I & 0 \\ A_\del \kron I_d & B_\del \kron I_d} \!\mat{c}{\zeta(t) \\ \t y_1(t)}\\
				&\geq -(\bullet)^T \big(P_0 \kron K \big) \mat{cc}{ 0 & I \\ C_\del \kron I_d & D_\del \kron I_d}\mat{c}{\zeta(t) \\ \t y_1(t)} \\
				&= -(\bullet)^T \big(P_0 \kron K \big)\mat{c}{\t y_1(t) \\ \t y_2(t)} = - (\bullet)^T\big(P_0 \kron K \big) \t y(t)
			\end{aligned}
		\end{equation*}
		for all $t \geq 0$. By integration, we obtain
		\begin{equation}
			\int_0^T (\bullet)^T \big(P_0 \kron K \big) \t y(t)\,dt \geq -(\eta(T) - \eta(0))  \geq 0
			\label{DSP::eq::K_inequality}
		\end{equation}
		for all $T\geq 0$, were we exploited $\eta(0) = 0$ and $\eta(T)\leq 0$. The latter two properties of $\eta$ are consequences of $\zeta(0) = 0$ as well as of $W \cle 0$ and of $K\cg 0$.

		Finally, integrating \eqref{DSP::eq::K_signal_expression} and using $\xi(0) = 0$ as well as \eqref{DSP::eq::K_inequality} yields
		\begin{multline*}
			\int_0^T y(t)^T\!\big(P_0 \kron M \big) y(t)\, dt + \xi(T)^T\! \big( P_0 \kron Y\big) \xi(T) \\
			= \int_0^T\t y(t)^T\!\big(P_0 \kron K \big) \t y(t)\, dt \geq 0
		\end{multline*}
		for all $T\geq 0$ as desired.
	\end{proof}
	% in the appendices.
\end{document}